\documentclass[final]{article}

\usepackage{amsmath,amssymb, amsthm}
\usepackage[mathcal]{eucal}             
\usepackage{mathrsfs}
\usepackage{graphicx}
\usepackage{floatflt}

\usepackage[latin1]{inputenc}
\usepackage[T1]{fontenc}

\providecommand*{\unit}[1]{\,\ifmmode \mathrm{\,#1}\else\textup{#1}\fi} 
\providecommand*{\ap}[1]{\ifmmode^\mathrm{#1}\else\textsuperscript{#1}\fi} 

\newcommand*{\newoperator}[3]{\newcommand*{#1}{\mathop{#2}#3}} 

\renewcommand\Re{\operatorname{Re}}
\renewcommand\Im{\operatorname{Im}}
\newoperator{\sgn}{\mathrm{sgn}}{\nolimits}

\newcommand*\dif{\mathop{}\!\mathrm{d}}
\newcommand*{\deriv}[3][]{\frac{\dif^{#1}{#2}}{\dif {#3}^{#1}}} 

\newoperator{\rot}{\mathrm{rot}}{\nolimits}

\newcommand*{\abs}[1]{\lvert #1 \rvert}
\newcommand*{\norm}[1]{\lVert #1 \rVert}

\DeclareMathAccent{\ring}{\mathalpha}{operators}{"17} 

\newoperator{\esup}{\mathrm{ess\,sup}}{\limits}
\newoperator{\spansym}{\mathrm{span}}{\limits}

\newcommand{\R}{\mathcal{R}}

\newcommand{\cb}{\mathcal {B}}
\newcommand{\ch}{\mathcal {H}}

\newcommand{\beq}{\begin{equation}}
\newcommand{\eeq}{\end{equation}}

\newcommand{\rone}{\mathbf {R}}

\newtheorem{theo}{Theorem}[section]

\newtheorem{lem}{Lemma}[section]

\newtheorem{defin}{Definition}
\newtheorem{remark}{Remark}[section]

\begin{document}

\title{Local energy decay for wave equation in the absence of resonance at zero energy in 3D}

\author{ Vladimir Georgiev\footnote{Department of Mathematics, Faculty of Sciences, Pisa University --- Largo Pontecorvo 5, 56127 Pisa, Italy, e-mail: \texttt{georgiev@dm.unipi.it}}\\%
Mirko Tarulli\footnote{Department of Mathematics, Faculty of Sciences, Pisa University --- Largo Pontecorvo 5, 56127 Pisa, Italy, e-mail: \texttt{tarulli@mail.dm.unipi.it}}}

\date{}

\maketitle

\begin{abstract}
In this paper we study spectral properties associated to Schr\"odinger operator $-\Delta-W,$ with potential  $W$ that is an exponential decaying $C^1$ function. As applications we prove local energy decay for solutions to the perturbed wave equation and lack of resonances for the NLS.
\end{abstract}

\section{Introduction} 

In this paper we study the problem of resonances at zero energy for the operator
\begin{equation}\label{eq.elliptic1}
S (\psi)(x) = -\Delta \psi(x) - W(|x|) \psi(x),
\end{equation}
with $W(r)$ being a positive real valued measurable function, decreasing sufficiently rapidly at infinity..
There exists a vast literature concerning the theory of resonances, we cite here \cite{A98}, \cite{BZ2001}, \cite{SZ97}, \cite{Sj2001} (and reference therein).
The resonances of an operator were introduced in physics and defined as the poles of its resolvent operator function taken in some generalized way. More precisely one can
observe that, if we choose a radial function $u(|x|)$ in $\mathbf{R}^3, $ we have the relation
$$ \triangle \left(  \frac{u(|x|)}{|x|} \right) = \frac{u^{\prime \prime}(|x|)}{|x|}.$$
Therefore, piking up $\psi(x)=u(x)/\abs{x}$ and $S(\psi)=P(u)/\abs{x},$ we can rewrite the operator \eqref{eq.elliptic1} as
\begin{equation}\label{eq.elliptic2}
P (u)(r) = -u^{\prime \prime}(r) - W(r) u(r)
\end{equation}
on the semi-line $(0,\infty)$ together with Dirichlet condition
$$ u(0) = 0 $$
at the origin, that is a selfadjoint unbounded operator in $L^2 (0,\infty)$ with domain
$$ \{ u \in H^2(0,\infty); u(0) = 0   \}. $$
It is well known that if the potential is of short range type, then the set of eigenvalues of $P$ is
finite, contained in $(-\infty ,0)$, with each eigenvalue of finite
multiplicity.

Recall that the  resolvent  of $P$
$$ R(\mu^2) =  (P-\mu^2)^{-1} $$
(considered as an operator from $C_0^\infty$ to $C^\infty$)  is a meromorphic operator in some subset of the complex plane. The poles of $R(\mu^2) $ are called resonances of $P$
The main goal of this work is to present an argument that gives sufficient condition for the non-existence
of resonances. This means the following,

\begin{theo}\label{theo.mainres}
Suppose the potential $W(r) \in C^1( 0, \infty)$ is a positive decreasing function  satisfying
(for some $r_0>0,$ $C>0$ and $ \varepsilon_0 > 0$ ) the estimates
\begin{equation}\label{eq:asWmain}
    |W^\prime(r)| + |W(r)|  \leq C e^{-\varepsilon_0 r}, \ \ r\geq r_0.
\end{equation}

Then there exists a positive $\delta$ such that there are no resonances in
$$  \{ \mu \in \mathbf{C};  |\mu| \leq \delta    \}.$$ Moreover, if $\mu $ with $\Im \mu > -\delta$ is a resonance,
then $\mu$ is a real negative number and $\mu$  is eigenvalue of $P.$
\end{theo}

\begin{remark}
The fact that all resonances in domain of type
$$ \{ \mu \in \mathbf{C}; 0 < |\mu| < R,  |\Im \mu| < \delta    \}$$
are eigenvalues is a well-known and follows from resolvent estimates leading to limiting absorption principle (see \cite{A} for example).
The fact that this domain can be extended taking $R \rightarrow \infty$  is also well known (see \cite{Ra66} for example).
Therefore, the key information in this theorem is the lack of resonances at the origin $\mu = 0.$
\end{remark}

The resonances can be considered in some manner like eigenvalues. The existence of non-trivial solution of the equation $Pu=0,$
is a typical obstacle to find dispersive properties of the time evolution group
associated with the Schr\"odinger operator
$-\Delta - W(x).$ Therefore, as a first application of the above theorem \ref{theo.mainres},
we shall look for dispersive properties to the solution of the following (see \cite{BPST2003}, \cite{BPST2004}, \cite{Sj2001}, \cite{Vai89} for further details),
\begin{equation}\label{NLW}
\begin{aligned} &
u_{tt}-\Delta u-Wu =0, \qquad (t,x)\in\mathbf{R}\times \mathbf{R}^3
\\&
  u(0,x)=u_0(x), \,  u_t(0,x)=v_0(x),  \end{aligned}
\end{equation}
where the potential $W$ satisfies the assumptions  \eqref{ass.res2}. In our setting, we get local energy decay, see Theorem \ref{th.localenergy}, for the above problem \eqref{NLW}.

Consider now the solution $u$ of the Schr\"odinger type equation
\begin{equation}\label{eq.Schro}
    i  \partial_t u
    +  \Delta u + |u|^{p-1} u=0.
\end{equation}

The existence of  solitary type solutions to the Schr\"odinger equation \eqref{eq.Schro} is well - studied problem.
One can see for example \cite{CL82} and \cite{St77} the existence results for  $ 1 < p \leq 1+ 4/3$.

The natural functional associated with this problem is
\begin{eqnarray}\label{eq.defEps}
    \mathcal{E} (\chi) = \frac{1}{2} \| \nabla \chi\|^2_{L^2} & - & \frac{1}{p+1} \int_{\mathbf{R}^3} |\chi|^{p+1} dx.
\end{eqnarray}

The corresponding minimization problem is associated with the
quantity
\begin{equation}\label{defImu}
    I_N = \inf \{ \mathcal{E}(\chi) ; \chi \in H^1, \|\chi\|^2_{L^2} = N
    \}.
\end{equation}

We have the following result, see \cite{CL82} and \cite{Ca2003} for more details,

\begin{lem} {\bf 1.} \label{l.parLions}
For any $\omega >0 $  there exists a
unique positive solution $\chi(x) = \chi_\omega (x)  \in H^1$ of
the equation \eqref{eq.M-S1},  such that

i) the function $ \chi(x) = \chi(|x|) $ is radial one,

ii) the function $ \omega \in (0,\infty) \rightarrow \|
\chi_\omega \|_{L^2} $ is strictly increasing one and belongs to $C^1(0,\infty) $.
\end{lem}

\begin{remark}\label{r.lio} To check ii) we use the following argument used in \cite{CL82}. Any positive radial solution of
$$ - \Delta \chi + \omega \chi = \chi^p , \ \chi \in H^1(\rone^3), \ \chi>0 $$
is given by $\chi(x) = \omega^{1/(p-1)} \chi_1(\sqrt{\omega} x) $ where $\chi_1$ is the unique radial solution of
$$ - \Delta \chi_1 +  \chi_1 = \chi_1^p , \ \chi_1 \in H^1(\rone^3), \ \chi_1>0 $$
Thus
$$ \|\chi\|_{L^2(\rone^3)} = \omega^{1/(p-1) - 3/4} \|\chi_1\|_{L^2(\rone^3)}$$
so $N = const \ \omega^{1/(p-1) - 3/4} $
where $1/(p-1) - 3/4>0.$ This yields the characterization of $N$ as a function of $\omega$.
\end{remark}

One can see that the solutions $\chi_\omega$ of the above Lemma are radial ones and $\chi_\omega(x) $ is rapidly decreasing in $x$ as $|x| \rightarrow \infty$ provided $N(\omega) \leq 1.$

 In particular the property ii) of the above Lemma guarantees that  one can find a unique $\chi^*$ that is radial positive function
 $ \chi_* = \chi_*(|x|)$ that is a minimizer of $I_1$ so there exists a unique $\omega_* >0$ so that

 \begin{eqnarray}\label{eq.M-S1star}
 \Delta \chi_* + \omega_* \  \chi_* =  \ \chi_*^p,& & \\
\int_{\mathbf{R}^3} |\chi_*|^2 &=& 1.\nonumber
\end{eqnarray}

A standard linearization of the Nonlinear Schr\"odinger equation  \eqref{eq.Schro} around the solitary solution leads to the necessity to use some spectral properties of the following
 operator
$$  -\triangle  - \chi_*^{p}(x).$$
Note that the operator
$$ L_- (\omega^*)= -\triangle  - \chi_*^{p}(x) + \omega^* $$
introduced in \cite{We85} (with rescaled choice $\omega^*=1$)  has a nontrivial kernel and plays important role in the study
of modulational stability of ground states of nonlinear Schr\"odinger equations.

It is well - known from the results in
\cite{CM2008}, \cite{GNT2004}, \cite{BUPe95}, \cite{BUSu95}, \cite{PeWe94}, \cite{Pe2004} that asymptotic stability around solitary waves is closely connected with the existence of resonances at the origin. More precisely, the following assumption
is frequently used in these articles:
\begin{equation}\nonumber
    (H1)
\begin{array}{ll}
   & \hbox{0 is not a resonance of $ -\triangle -  \chi_*^{p}(x)
    $   .}
\end{array}
\end{equation}

The main goal of this work is to present an argument that proves the assumption $(H1)$ in the general case and therefore the above cited results can be established without this additional assumption
and this is contained in Theorem \ref{Th:NLSres}.

The scheme of the paper is the following: in Section 2 we distinguish between strong and weak resonances, defining the last one. We also give an asymptotic expansions of the corresponding solution.
In Section 3 we give the main Theorem \ref{th:nores} on lack of strong resonances at the spectral point zero, in the radial case. In Section 4 we extend the result to the non-radial case. In Sections 5 and 6 we obtain some application, local energy decay for wave equation perturbed by a potential (Theorem \ref{th.localenergy}), and lack of resonance (and eigenvalue) for the linearized operator of NLS around its ground states (Theorem \ref{Th:NLSres}).
Finally in Section 7 we furthermore focalize on the weak resonances.

We will set $H^{m,s}$
$$\|u\|_{H^{p,s}}=\| \langle \Delta\rangle
^su\|_{H^m(\R^2)},$$
 where $p\in \mathbf{R}$, $s\in \mathbf{R}$ and $\langle \Delta\rangle =(1+|\Delta|^2)^{1/2}$
Moreover, given any two positive real numbers $a, b,$ we write $a\lesssim b$ to indicate
$a\leq C b,$ with $C>0.$



\section{Resonance at the spectral origin}

To study the poles of perturbed resolvent $R(\mu^2) = (P - \mu^2)^{-1}$ we start with explicit representation
of the free resolvent $R_0(\mu^2) = (P_0 - \mu^2)^{-1},$ where
$$ P_0(u)(r) = -u^{\prime \prime}(r) $$
with $r \in (0,\infty)$ together with Dirichlet condition
$$ u(0) = 0 $$
at the origin. Then
\begin{equation}\label{eq.rep}
 R^{\pm}_0(\mu^2) (f) (r) = \int_0^r e^{\pm i \mu r} \frac{\sin (\mu s)}{\mu} f(s) ds +
 \int_r^\infty e^{\pm i \mu s} \frac{\sin (\mu r)}{\mu} f(s) ds,
 \end{equation}
are well defined for $f \in C_0^\infty.$
A direct calculation shows that $$u_{\pm}(r) = R^{\pm}_0(\mu^2) (f) (r)$$ satisfies the equation
$$P_0(u) - \mu^2 u = f$$
as well as the Dirichlet condition $u(0) = 0.$ Choosing the sign $+$ in the above representations we see that
$$ R_0(\mu^2) (f) (r) = \int_0^r e^{ i \mu r} \frac{\sin (\mu s)}{\mu} f(s) ds +
 \int_r^\infty e^{ i \mu s} \frac{\sin (\mu r)}{\mu} f(s) ds $$
can be extended as an operator in $L^2$ provided $ \Im \mu > 0.$  Moreover, this is a holomorphic operator-valued function for $ \Im \mu > 0.$ The representation formula guarantees also that the operator
$$ \varphi (r)  (P_0 - \mu^2)^{-1} \varphi (r),  \ \ \varphi(r) = e^{-\delta r}$$
is holomorphic operator valued (in $L^2$) function in larger domain $ \Im \mu > -\delta.$

The perturbed resolvent $R(\mu^2) = (P - \mu^2)^{-1}$ satisfies in $ \Im \mu > 0$ the relation
\begin{equation}\label{eq.pert1}
(P - \mu^2)^{-1} = \left(I - (P_0-\mu^2)^{-1} W     \right)^{-1} (P_0-\mu^2)^{-1}
\end{equation}
provided the operator $ (I - (P_0-\mu^2)^{-1} W) $ is invertible. This relation implies
\begin{equation}\label{eq.pert2}
 \varphi (P - \mu^2)^{-1} \varphi  = \left(I - \varphi (P_0-\mu^2)^{-1} W \frac{1}{\varphi}    \right)^{-1} \varphi (P_0-\mu^2)^{-1} \varphi .
 \end{equation}
Applying the Fredholm alternative, we see that
$$ \varphi (P - \mu^2)^{-1} \varphi $$ is holomorphic except the points $\mu$ such that
the equation
$$ f = \varphi (P_0-\mu^2)^{-1} W \frac{f}{\varphi}$$
has a nontrivial solution $f \in L^2(0,+\infty).$ These complex numbers $\mu$ are the resonances and we give a more detailed description of the resonances in the following.
\begin{lem} \label{l;strres} Assume $W(r) \leq C e^{-\varepsilon_0 r}$ and $\varphi(r) = e^{-\delta r}$ with $0 < \delta < \varepsilon_0/2.$ For any complex $\mu$ with $0 \geq \Im \mu > - \delta$  the following conditions are equivalent:

i) $\mu$  is a resonance, i.e. there is $f \in L^2(0,\infty)$ so that $f$ is not identically zero and
$$ f = \varphi (P_0-\mu^2)^{-1} W \frac{f}{\varphi};$$

ii) there is $u \in C[0,\infty)$ not identically zero, $e^{Im \mu \,  r} u $ is bounded and
$$ u =  (P_0-\mu^2)^{-1} W u;$$

iii) there is $u \in C[0,\infty) \cap C^2(0,\infty)$ so that $u$ is not identically zero, $e^{Im \mu \,  r}u$ is bounded and $u$ solves
$$ Pu = \mu^2u, \ \ u(0) = 0.$$
\end{lem}

\begin{proof}
$i) \Leftrightarrow ii).$ The proof follows  from standard substitution $u=f\varphi^{-1}$ and classical estimates. We get
$ u =  (P_0-\mu^2)^{-1} W f\varphi^{-1}.$ Now by the assumption on the complex number $\mu,$
we can write
$$\norm{e^{Im \mu} u}_{L^\infty}\lesssim \norm{u}_{L^\infty}\lesssim \norm{(P_0-\mu^2)^{-1}W f\varphi^{-1}}_{L^\infty}.$$ The resolvent estimates $\norm{(P_0-\mu^2)^{-1} g}_{L^\infty}\lesssim \norm{\langle{x\rangle}^{1/2+\varepsilon}g}_{L^2},$ (see \cite{GV2003} and reference therein),  and the bound $\langle{x\rangle}\lesssim \varphi^{-1},$ yield the  implication $ i) \Rightarrow ii).$  The other implication follows from
$$ |f(x)| = |\varphi(x) u(x)| \leq e^{-(\delta+\Im \mu)|x|} \in L^2  $$.\\
$ii) \Rightarrow iii).$  Applying
the operator $P_0$ to both sides of the identity $u =  (P_0-\mu^2)^{-1} W u,$ we get the result. Moreover the representation \eqref{eq.rep} says that $u(0)=0.$\\
$iii) \Rightarrow ii).$ The proof follow by an application of the Limiting Absorption Principle (see \cite{A}) and by the bound $\varphi\lesssim\langle{x\rangle}^{-1}.$
\end{proof}

The above Lemma reduces the study of resonance to the study of the solutions to the problem
$$ Pu = \mu^2u, \ \ u(0) = 0,$$
satisfying the bound
$$ |u(r)| \leq C e^{-\Im \mu \,  r} , \ \  0 \geq \Im \mu > - \delta.$$

More general question to solve is the existence of nontrivial solutions satisfying weaker bound
$$ |u(r)| \leq C e^{\delta \,  r} , \ \  0 \geq \Im \mu > - \delta.$$

We distinguish these two cases: we call the resonances from Lemma \ref{l;strres} strong resonances and define the weak ones as follows:
\begin{defin}\label{defresliweak}
A complex number $\mu$ with $ 0 \geq \Im \mu > - \delta $ is called a weak resonance of \[P = -\left(\deriv{}{r}\right)^2 - W(r) \] if there exists $u \in C([0, \infty)),$ such that $u(0)=0$,  $u(r)$ is not identically zero, $P(u)=\mu^2 u$  in distribution sense in $(0,\infty)$ and the solution $u$ satisfies the inequality
\begin{equation}\label{eq.1asu}
    |u(r)| \leq C \, e^{\delta |r|} .
\end{equation}
\end{defin}
We recalling that, as underlined for example in \cite{CGNT2007} and \cite{Pe2001}, that the resonances are defined as functions not in $L^2$ but in larger $L^2$ weighted spaces. Moreover these functions are $L^\infty$-bounded and behave asymptotically as $\langle{x\rangle}^{-1},$ in spatial dimension three. These resonances will be called strong resonances (see Definition \ref{defreslistrong} for precise notion). But, for completeness, we study also the weak ones.

Moreover, as it was mentioned in the introduction, we have to study only the existence of resonance at the origin.

\section{Strong resonance at the spectral origin.}

\begin{defin}\label{defreslistrong}
A real number $\lambda$ is called a  strong  resonance of \[P = \left(\deriv{}{r}\right)^2 + W(r) \] if there exists $u \in C([0, \infty)),$ such that $u(0)=0,$ $u(s)$ is not identically zero, $P(u)=\lambda u$  in distribution sense in $(0,\infty)$ and the solution $u$ satisfies the inequality
\begin{equation}\label{eq.1asust}
    |u(r)| \leq C \, (1+r)^a .
\end{equation}
with some $0 \leq a < 1.$
\end{defin}

We shall need the asymptotic expansions of Lemma \ref{lem:2}.
Without loss of generality we can assume $u$ is real - valued.
Multiplying the equation
$$ P(u) \equiv u^{\prime \prime}(r) + W(r) u = 0 $$
by $ u^\prime$ and integrating over $(r, \infty),$ we find
$$ - \frac{|u^\prime(r)|^2}{2}  - \frac{W (r)|u(r)|^2}{2} - \int_r^\infty \frac{W^\prime(\tau)|u(\tau)|^2}{2} \ d\tau =0.$$
Take any  function  $g(r) \in C([0,\infty))$,  such that   $g(r)$ tends to $0$  at infinity.
We multiply further the last relation by $-g(r)$ and integrate over $(0,\infty)$
\begin{eqnarray}\label{eq.4rel}
    \int_0^\infty \frac{g(r) |u^\prime (r)|^2}{2}\  dr  +   \int_0^\infty \frac{g(r) W(r)|u(r)|^2}{2}\  dr +
    \\ \nonumber +
    \int_0^\infty \frac{G(r) W^\prime(r) |u (r)|^2}{2}\  dr =0,
\end{eqnarray}
where
$G(r) = \int_0^r
g(\tau) d \tau.$

Take a $C^1$ function  $h(r)$ on $(0,\infty)$ such that $h^{\prime
\prime}(r)$ exists, continuous on $(0,r_0) \cup (r_0, +\infty)$
and has a finite jump
 $$ \lim_{r \rightarrow r_0^+} h^{\prime \prime} (r) - \lim_{r \rightarrow r_0^-} h^{\prime \prime} (r) < \infty.$$
 We shall require further that $h^\prime(r)$ tends to zero at infinity and $h^{\prime\prime}(r)$ is integrable on $(0,\infty).$
 We multiply the equation $Pu=0$ by $hu$
and integrate over  $(0, \infty),$ so we get
\begin{eqnarray}\label{eq.5rel}
    \int_0^\infty \frac{h^{\prime \prime}(r) |u (r)|^2}{2}\  dr  - \int_0^\infty h(r) |u^\prime(r)|^2\  dr +
    \\ \nonumber +
    \int_0^\infty h(r) W(r) |u (r)|^2\  dr =0.
\end{eqnarray}
Choosing $h = g/2$ and summing the above two relations, we obtain
\begin{eqnarray}\label{eq.6rel}
     \int_0^\infty \Phi(r) \  |u(r)|^2\  dr
    =0,
\end{eqnarray}
where
$$ \Phi (r) = 2g(r)W(r) + G(r) W^\prime(r) +  \frac{g^{\prime \prime}(r)}{2} . $$
The starting definition of the function $ g^{\prime \prime}(r) $ is the following one
\begin{equation}\label{eq.defgpp}
   g^{\prime \prime}(r) = \left\{
                            \begin{array}{ll}
                               r^M, & \hbox{if $0 \leq r < r_0$;} \\
                               e^{-\delta r}, & \hbox{if $r > r_0$.}
                            \end{array}
                          \right.
\end{equation}

We take $r_0=M^2$ and shall define the positive parameters $ M,
\delta$ later on.

For $0 \leq r < r_0$ we have
$$ g(r) = \frac{r^{M+2}}{(M+2)(M+1)}, \  G(r) = \frac{r^{M+3}}{(M+3)(M+2)(M+1)} .$$
From these relations we deduce the following.

\begin{lem} \label{lem:middlereg}
Suppose the potential $W(r)$ is a positive decreasing function, such that  the assumption \eqref{eq:asWmain} is satisfied.
 Then one can find a positive
constant $M^*$ depending on $ C^*, \varepsilon_0, $ such that for
any $r \in  (0, M^2)$ we have
$$ \Phi(r) \geq 2 g W.$$
\end{lem}

\begin{proof}
We have the estimates
$$ G(r) W^\prime(r) +  \frac{g^{\prime \prime}(r)}{2} \geq  - \frac{ C^* r^{M+3} e^{-\varepsilon_0 r}}{(M+3)(M+2)(M+1)}  +$$ $$   + \frac{r^M}{2} =\frac{r^{M}}{2} \left(1 -  \frac{ 2C^* r^{3} e^{-\varepsilon_0 r}}{(M+3)(M+2)(M+1)}    \right) .   $$
Now we are in position to choose $M^*=M^*(C^*, \varepsilon_0)$ so large that
$$ 1 -  \frac{ 2C^* r^3 e^{-\varepsilon_0 r}}{(M+3)(M+2)(M+1)} \geq 0$$
for $r  \leq M^2$ and $M \geq M^*.$ Hence
$$ \Phi (r) = 2g(r)W(r) + G(r) W_2^\prime(r) +  \frac{g^{\prime \prime}(r)}{2} \geq 2g(r)W(r).$$
This completes the proof.
\end{proof}

For $ r >  r_0 = M^2$ we shall use the relations
$$ g(r) = \int_0^r (r-\tau) g^{\prime \prime}(\tau) d \tau, \ \ G(r) = \int_0^r \frac{(r-\tau)^2}{2} \ g^{\prime \prime}(\tau) d \tau. $$
Here and below we used the property $g^\prime(0)=0.$

\begin{lem} \label{lem:largereg}
Suppose the potential $W(r)$ is a positive decreasing function, such that there exist  positive constants
$C^*, \varepsilon_0$ so that
$$ \left| W(r)\right| + \left| W^\prime(r)\right| \leq C^*e^{-\varepsilon_0 r} . $$
Then one can find a positive constant $M^*$ depending on $ C^*, \varepsilon_0, $ such
that for any $M \geq M^*$  and for any $r \in  (M^2, \infty)$ we have
$$ \left| \Phi(r) \right|\leq 2 e^{-\varepsilon_0 r/2}.$$
\end{lem}

\begin{proof}
Take $r > M^2.$ First we  evaluate
$$ g(r) = \int_0^{M^2} (r-\sigma) \sigma^{M} d \sigma + \int_{M^2}^r (r-\sigma) e^{-\delta \sigma} d\sigma \leq $$
$$\leq  r \int_0^{M^2} \sigma^{M} d \sigma + r \int_{0}^\infty  e^{-\delta \sigma} d\sigma \leq r( M^{2M+2} + \frac{1}{\delta}) \leq \frac{2 r^2 r^{\sqrt{r}}}{\delta}$$
for $\delta < 1$ and $M>1.$
In a similar way we get
$$ 2G(r) = \int_0^{M^2} (r-\sigma)^2 \sigma^{M} d \sigma + \int_{M^2}^r (r-\sigma)^2 e^{-\delta \sigma} d\sigma \leq $$
$$\leq  r^2 \int_0^{M^2} \sigma^{M} d \sigma + r^2 \int_{0}^\infty  e^{-\delta \sigma} d\sigma \leq r^2( M^{2M+2} + \frac{1}{\delta}) \leq \frac{2 r^3 r^{\sqrt{r}}}{\delta}.$$
From these estimates and the assumptions on the decay of $W$ we derive
$$|\Phi(r)| = |2g(r)W(r) + G(r) W_2^\prime(r)+\frac{g^{\prime \prime}(r)}{2}|  \leq \frac{4 C^* e^{-\varepsilon_0 r} r^2 r^{\sqrt{r}}}{\delta}+ e^{-\delta r}. $$
Taking $\delta = \varepsilon_0/2$ and using the fact that
$$ \lim_{r \rightarrow \infty} e^{-\varepsilon_0 r/2} r^2 r^{\sqrt{r}} = 0$$
we find $M^*=M^*(C^*,\varepsilon_0)$ so that
$$ |\Phi(r)| \leq 2 e^{-\varepsilon_0 r/2}$$ for $r>M^2$ and $M>M^*.$
This completes the proof of the Lemma.
\end{proof}

Now we can state in a precise form the main Theorem \ref{theo.mainres}:

\begin{theo} \label{th:nores}
Suppose the potential $W(r)$ is a positive decreasing function, such that there exist  positive constants
$C^*, \varepsilon$ so that
\begin{equation}\label{ass.res2}
 \left| W(r)\right| + \left| W^\prime(r)\right| \leq C^*e^{-\varepsilon_0 r} .
\end{equation}
with some $ \varepsilon_0 >0.$
The zero is not a strong resonance for $Pu = u^{\prime \prime}(r) + W(r) u(r).$
\end{theo}

\begin{proof}
Suppose $u(r)$ is a  solution to $u^{\prime \prime}(r) + W(r)
u(r)=0,$ such that $u (r)$ is not identically zero, i.e. $u(r_0) \neq 0.$ We lose no generality assuming $r_0=1$, so $u(1) \neq 0$.  We choose  $M^*$
so that the conclusions of  Lemma \ref{lem:middlereg} and Lemma
\ref{lem:largereg} are fulfilled. Then the identity \eqref{eq.6rel}
implies that
\begin{equation}\label{eq.conres}
    \int_0^{M^2} \Phi(r) \  |u(r)|^2\  dr
    = - \int_{M^2}^\infty \Phi(r) \  |u(r)|^2\  dr
\end{equation}
for any $M \geq M^*.$ We can apply Lemma \ref{lem:2} and conclude that
$$   u(s) = u(1) + \int_{1}^s(\tau-1) W(\tau) u(\tau) d\tau + \int_s^\infty ( s-1) W(\tau) u(\tau) d\tau,$$
This relation and the exponential decay of the potential $W(\tau)$ implies
$$S = \sup  |u(s)| \leq  |u(0)| + \int_{0}^\infty (\tau+1) W(\tau)
|u(\tau)| d\tau  < \infty.$$ Applying
Lemma \ref{lem:middlereg}, see that
$$ \int_0^{M^2} \Phi(r) \  |u(r)|^2\  dr \geq \int_0^{M^2} g(r) W(r) \  |u(r)|^2\  dr \geq $$ $$ \geq \int_1^{2} g(r) dr W(1) \  |u(1)|^2 = \frac{2^{M+3}-1}{(M+3)(M+2)(M+1)} \ W(1) \  |u(1)|^2 \geq W(1) \  |u(1)|^2.$$
The right hand side of the identity \eqref{eq.conres} can be evaluated by the aid of Lemma \ref{lem:largereg}  and we get
 $$ - \int_{M^2}^\infty \Phi(r) \  |u(r)|^2\  dr \leq 2 \int_{M^2}^\infty e^{-\varepsilon_0 r/2} |u(r)|^2 dr \leq
2 S^2 \ \frac{e^{-\varepsilon_0 M^2/2}}{\varepsilon_0}$$ so we
arrive at the inequality
$$ W(1) \  |u(1)|^2 \leq 2 S^2 \frac{e^{-\varepsilon_0 M^2/2}}{\varepsilon_0}$$
that obviously leads to a contradiction with $u(1) \neq 0$ if
$M>M^*$ is sufficiently large. The contradiction implies $u=0.$
This completes the proof.

\end{proof}

\section{The non-radial case: zero is neither an eigenvalue nor a resonance.}
Along the previous sections we treated the non-existence of radial resonances. Our next step is to treat the general case, that is the lack of resonances in the nonradial case. One can use standard projections on spherical harmonics and reduce the analysis to the proof that zero is not resonance for the following operator (see \cite{CGNT2007} and \cite{We85}),

\begin{equation}\label{eq.Pu}
P(u)(r) = - u''(r) - W(r)u(r),   \quad r \in (0, +\infty).
\end{equation}
Here  and below we  shall assume that  $W(r) = W_1(r)+W_2(r),$ where
\begin{equation}\label{eq.w1}
    W_1(r) = -\frac{\alpha (\alpha + 1)}{r^2},  \quad \alpha \geq 0,
\end{equation}
while $W_2(r) $ is a $C^1(0, \infty)$ positive strictly decreasing function such that for some positive constants $C, \varepsilon_0$ satisfies the estimate
  \begin{alignat}{2}
    & |W_2(r)| < C e^{-\varepsilon_0 r} , & \qquad & \text{for any } r>0 . \label{eq.h2b}
\end{alignat}
It is clear that we need for the applications only the case, when $\alpha(\alpha+1)$ is an eigenvalue of the Laplace-Beltrami operator on the sphere $\mathbf{S}^2.$ The arguments from this section are valid for any $\alpha>0$

One can have
\begin{defin} \label{eq.h2sob}
A real number $\lambda$ is called eigenvalue of $P$ if there exists $u \in H^1(0, \infty)$ such that $u(0)=0,$ $u(r)$ is not identically zero and $P(u)=\lambda u$  in distribution sense in $(0, +\infty).$
\end{defin}

The first step is to show that $0$ is not an eigenvalue. This means the following:

\begin{theo} \label{l.onee}
$\lambda=0$ is not an eigenvalue of $P.$
\end{theo}

 \begin{proof} Suppose that there exists a real valued function $u(r) \in H^1(0,\infty)$ so that
$P(u)=0.$ Our goal will be to show that $u$ is identically zero.

The Sobolev embedding on $(0,\infty)$ implies that $u(r) \in C([0,\infty)).$ Then the equation $Pu=0$ guarantees that
$ u \in H^2 (R, \infty) \subset C^1( [R, \infty))$ for any $R >0.$ To analyze the behavior of the solution at infinity, we integrate the
equation $Pu=0$ in the interval $(R, R_1)$ and find
 \[\left|  u'(R) - u'(R_1)   \right| \leq \left(\int_R^{R_1} W(r)^2 \dif r \right)^{1/2} \|u\|_{L^2}
\leq C R^{-3/2}, \] since at infinity $W(r)$ behaves like $r^{-2}.$ The assumption $u \in H^1(0,\infty)$ easily yields
\begin{equation}\label{eq.o2}
   \left|  u'(R)   \right| \leq CR^{-3/2}, \qquad \left|  u(R)   \right| \leq CR^{-1/2}.
\end{equation}
From the asymptotic expansion obtained in the previous Lemma we have also
\begin{equation}\label{eq.o2fine}
   \left|  u'(R)  + \frac{\alpha}R   u(R)\right| \leq Ce^{-\varepsilon_0 R/2}.
\end{equation}

One can use the relation \eqref{eq.6rel} taking into account that
$$ \Phi (r) = \Phi_1(r) +  2g(r)W_2(r) + G(r) W_2^\prime(r) +  \frac{g^{\prime \prime}(r)}{2}  $$
with
$$\Phi_1(r) = -\alpha(\alpha+1) \left( \frac{2g(r)}{r^2} -  \frac{2G(r)}{r^3} \right).$$

Then one can proceed as in the proof of Theorem \ref{th:nores} modifying the assertions of Lemmas
as follows

\begin{lem} \label{lem:middleregnonr}
Suppose the potential $W_2(r)$ is a positive decreasing function, such that  the assumption \eqref{eq:asWmain} is satisfied.
 Then one can find a positive
constant $M^*$ depending on $ C^*, \varepsilon_0, $ such that for
any $r \in  (0, M^2)$ we have
$$ \Phi (r) = \Phi_1(r) +  2g(r)W_2(r) + G(r) W_2^\prime(r) +  \frac{g^{\prime \prime}(r)}{2} \geq 2 g W_2.$$
\end{lem}

\begin{lem} \label{lem:largeregnonr}
Suppose the potential $W_2(r)$ is a positive decreasing function, such that  the assumption \eqref{eq:asWmain} is satisfied.
Then one can find a positive constant $M^*$ depending on $ C^*, \varepsilon_0, $ such
that for any $M \geq M^*$  and for any $r \in  (M^2, \infty)$ we have
$$ \left| \Phi(r) \right|\leq 2 e^{-\varepsilon_0 r/2}.$$
\end{lem}

For the proof of the first Lemma it is sufficient to recall that dominant term
in $$ 2g(r)W_2(r) + G(r) W_2^\prime(r) +  \frac{g^{\prime \prime}(r)}{2} $$
for $0 \leq r \leq M^2$
is $$ \frac{g^{\prime \prime}(r)}{2} =  \frac{r^M}{2}. $$

Since
$$\Phi_1(r) = -\frac{2r^M \alpha(\alpha+1)}{(M+3)(M+1)} ,$$
we see that for $M$ large enough
$$  \Phi (r)  \geq 2 g W_2.$$

This completes the proof of the Lemma.
\end{proof}

Once we have proved lack of zero-energy eigenvalue, we shall prove the following:

\begin{lem} \label{l.twores}
Suppose $\alpha>1/2$. Then $\lambda=0$ is not a strong resonance of $P.$
\end{lem}
\begin{proof}
According to the asymptotic formula \eqref{eq.O7l}, if the operator $P$ has in the spectral point $\lambda=0$ a resonance, then $u(r)$ is a function in $L^2(1, \infty),$ but this is clearly a contradiction. This easily concludes the proof of the lemma.
\end{proof}

Finally, we may study the resonances of the
operator
$$-\Delta - W(|x|), x \in \R^3.$$

\begin{defin}\label{defreslistrongR3}
A real number $\lambda$ is called a  strong  resonance of $ -\Delta - W(|x|) $ if there exists $u \in C(\R^3),$ such that  $u(x)$ is not identically zero, $-\Delta u - W(|x|)u=\lambda u$  in distribution sense in $\R^3$ and the solution $u$ satisfies the inequality
\begin{equation}\label{eq.1asustRn}
    |u(x)| \leq C \, (1+|x|)^{-\varepsilon} .
\end{equation}
with some $\varepsilon > 0.$
\end{defin}

\begin{theo} \label{th:noresR3}
Suppose the potential $W(r)$ is a positive decreasing function, such that there exist  positive constants
$C^*, \varepsilon$ so that \eqref{ass.res2} is fulfilled.
Then zero is not a strong resonance for $-\Delta -W(|x|).$
\end{theo}

\begin{remark}
Since $W$ is an exponentially decaying and real valued, the above result implies that $-\Delta -W(|x|)$ has no resonances
.
\end{remark}

\section{Resolvent estimates and local energy decay for wave equation with potential}
Along this section we will prove the main resolvent estimates concerning the perturbed operator \eqref{eq.elliptic1}. Let us indicate by
$$
R_0(\mu)=(-\Delta-\mu^2)^{-1},
$$
the resolvent of the operator $-\Delta,$ and set $R^{+}_0(\mu)=R_0(\mu)$ if $\Im\mu>0$ and respectively $R^{-}_0(\mu)=R_0(\mu)$ for $\Im\mu<0.$ Take into account now the initial value problem
\begin{equation}\label{eq.lwa1}
\begin{cases}
\partial^2_t u - \Delta u =0\ \ \  t \in \mathbf{R} , \ \ \ x \in \mathbf{R}^3\\
(u(0,x), \partial_tu(0,x))=(0,g(x)),
\end{cases}
\end{equation}
Once we pick $\Im\mu>C>0,$ we have, by an application of the Laplace transform,
\begin{equation}\label{eq.laplfree}
R_0(\mu)=\int_0^\infty e^{i\mu t}\mathcal{U}_0(t)dt
\end{equation}
where $\mathcal{U}_0(t)=\frac {\sin t\sqrt {-\Delta}}{\sqrt {-\Delta}}$ is the evolution operator associated to \eqref{eq.lwa1}. This means that the resolvent associated to \eqref{eq.lwa1} is a well-defined operator in $\cb(H^{m-1}, H^m),$ and depends analytically by $\mu,$ once one notice that (see \cite{KT}  and \cite{Sj2001} to have more details)
\begin{equation}\label{eq.freeboound}
\norm{\mathcal{U}_0(t)f}_{H^m}\lesssim \norm{f}_{ H^{m-1}}.
\end{equation}
for any Schwartz function $f.$ This aims to the following inequality, after an integration by parts of \eqref{eq.laplfree} and by \eqref{eq.freeboound}, to the bound,
\begin{equation}\label{eq.freest2}
\norm{R^{+}_0(\mu)f}_{H^m}\lesssim \frac{1}{\langle{\mu\rangle}}\norm{f}_{ H^{m}},
\end{equation}
for any Schwartz $f.$ Moreover one could get from the identity,
$$
R^{+}_0(\mu)=-\frac{1}{\mu^2}-\frac{1}{\mu^2}\Delta R^{+}_0(\mu),
$$
the estimate,
\begin{equation}\label{eq.freest3}
\norm{R^{+}_0(\mu)f}_{L^2}\lesssim \frac{1}{\langle{\mu\rangle}^2}\norm{f}_{ H^{1}}.
\end{equation}

Other relevant estimates obtained easily from the \eqref{eq.laplfree} are,
\begin{equation}\label{eq.freest4}
\norm{\frac{d}{d\mu}R^{+}_0(\mu)f}_{L^2}\lesssim \frac{1}{\langle{\mu\rangle}^2}\norm{f}_{ L^2},
\end{equation}
\begin{equation}\label{eq.freest5}
\norm{\frac{d^2}{d\mu^2}R^{+}_0(\mu)f}_{L^2}\lesssim \frac{1}{\langle{\mu\rangle}^3}\norm{f}_{ L^2},
\end{equation}
and
\begin{equation}\label{eq.freest6}
\norm{\frac{d^2}{d\mu^2}R^{+}_0(\mu)f}_{L^2}\lesssim \frac{1}{\langle{\mu\rangle}^2}\norm{f}_{H^{-1}}.
\end{equation}
Recall the classical resolvent identities \eqref{eq.pert1}, \eqref{eq.pert2}
and set
$$R(\mu)=  (-\Delta-W-\mu^2)^{-1} , \ \ A(\mu)=\varphi (-\Delta-\mu^2)^{-1} W \varphi^{-1}.$$
One have  the following compactness result in the spaces
$$ L^{2}_{-\delta} = \{ f \in L^2_{loc}, (1+|x|)^{-\delta} f \in L^2  \}.$$


\begin{lem}\label{comp1}
The operators $A(\mu)$ are compact in the space $B(L^{2}_{-\delta},L^{2}_{-\delta'}),$ for $\Im \mu>-\delta$ and $\delta, \delta'>0$. Moreover the following estimate is satisfied: $$\|A(\mu)\|_{B(L^{2}_{-\delta},L^{2}_{-\delta'})}\rightarrow 0,$$ as $\abs{\mu}\rightarrow \infty$.
\end{lem}
This Lemma is a well-known standard result  so we skip the proof. It suffices notice that $\varphi (-\Delta-\mu^2)^{-1} $W is analytic on the zone $\Im \mu>-\delta,$ and that the potential $W$ is of the short range type (see \cite{A} and references therein). The continuity of the multiplication operator $\varphi^{-1}$
in $B(L^{2}_{-\delta},L^{2}_{-\delta'}),$ with the estimate \eqref{eq.freest2}, give the result.

\begin{lem}\label{comp2}
Let us assume that the potential $W$ satisfies  \eqref{ass.res2}. The cutoff resolvent operator $\varphi R^{+}(\mu)\varphi$ has a meromorphic extension from $\Im \mu>0$ to $\Im \mu>-\delta.$ Moreover for each $\delta>0$ there exists a real constant $C>0$ such that the following estimates are true:
\begin{equation}\label{finf1}
  \|\varphi P_{ac} R^{+}(\mu)\varphi f\|_{L^2} \leq \frac{C}{{\langle{\mu\rangle}}}
  \|f\|_{L^{2}},
\end{equation}
\begin{equation}\label{finf2}
  \|\varphi P_{ac}R^{+}(\mu)\varphi f\|_{L^2} \leq \frac{C}{{\langle{\mu\rangle}^2}}
  \|f\|_{H^{1}},
\end{equation}
\begin{equation}\label{finf3}
  \|\varphi \frac{d}{d\mu}P_{ac} R^{+}(\mu)\varphi f\|_{L^2} \leq C\frac{C}{{\langle{\mu\rangle}^2}}
  \|f\|_{L^2},
\end{equation}
\begin{equation}\label{finf4}
  \|\varphi \frac{d^2}{d\mu^2}P_{ac} R^{+}(\mu)\varphi f\|_{L^2} \leq C\frac{C}{{\langle{\mu\rangle}^2}}
  \|f\|_{H^{-1}},
\end{equation}
for any Schwartz function $f.$
\end{lem}
\begin{proof}
We start to prove the first claim. Denote by $P_{ac}$ the projection on the absolutely continuous part of the operator
$$ -\Delta - W.$$
The perturbed resolvent $R(\mu^2) = (P - \mu^2)^{-1}$ satisfies in $ \Im \mu > 0$ the relation
\begin{equation}\label{eq.pert1dd}
P_{ac}(-\Delta - W - \mu^2)^{-1} = P_{ac}\left(I - (\Delta-\mu^2)^{-1} W     \right)^{-1} (-\Delta-\mu^2)^{-1}
\end{equation}
provided the operator $ (I - (P_0-\mu^2)^{-1} W) $ is invertible. This relation implies
\begin{equation}\label{eq.pert2dd}
 \varphi P_{ac} (-\Delta - W - \mu^2)^{-1} \varphi  = \varphi P_{ac} \varphi^{-1}\left(I - \varphi (P_0-\mu^2)^{-1} W \frac{1}{\varphi}    \right)^{-1} \varphi (P_0-\mu^2)^{-1} \varphi .
 \end{equation}
 We can assume that $\mu_1<  \cdots, < \mu_N < 0$ are the eigenvalues of the operator $-\Delta - W$ with corresponding eigenvectors (normalized in $L^2$) eigenvectors $$f_1(x)<  \cdots, < f_N(x) $$ They decay exponentially and this fact implies
 $$ (I- P_{ac})f = \sum_{j=1}^N f_j \langle f, f_j \rangle_{L^2}. $$
 Hence taking $\delta >0$ sufficiently small in $\varphi(x)=e^{-\delta|x|}$ we get
 $$\| ( I- P_{ac}) \varphi^{-1} f \|_{L^2} \leq C \|  f \|_{L^2}.$$ Hence the operator $  \varphi P_{ac} \varphi^{-1}$ is bounded in $L^2$ and the  relation \eqref{eq.pert2dd} can be used in combination with
Analytic Fredholm Theory and the Theorem \ref{th:nores} concerning the lack of strong resonances in $\mu=0,$
we are able to say that the operator
$$
\left(I - \varphi (-\Delta-\mu^2)^{-1} W \frac{1}{\varphi}    \right)^{-1},
$$
is analytic in $\Im \mu>-\delta,$ excluded a discrete subset where there are the eigenvalues of \eqref{eq.elliptic1}.
Moreover, the application of Theorem \ref{th:noresR3} and the remark after the Theorem gurantees that
$$ \varphi P_{ac} \varphi^{-1} \left(I - \varphi (-\Delta-\mu^2)^{-1} W \frac{1}{\varphi}    \right)^{-1} $$
is analytic in $\Im \mu>-\delta.$

In this way we obtain the inequality
$$
 \norm{\varphi P_{ac}R^{+}(\mu) \varphi f}_{L^{2}}  \lesssim $$

$$ \lesssim \norm{\varphi P_{ac} \varphi^{-1}\left(I - \varphi (-\Delta-\mu^2)^{-1} W \frac{1}{\varphi}    \right)^{-1}}_{B(L^{2}_{-\delta},L^{2}_{-\delta'})}  \norm{\varphi R_0^{+}(\mu) \varphi f }_{L^{2}}.$$
Considering $\varphi \in L^1(\rone),$ the right hand side of the previous estimate could be bounded in several different way: by \eqref{eq.freest2}, it does not exceed
$C\langle{\mu\rangle}^{-1}
  \| f\|_{L^{2}},$ while, from \eqref{eq.freest3}, we get that it is less than $C\langle{\mu\rangle}^{-2}
  \|f\|_{L^{2}}.$ In that way the resolvent estimates \eqref{finf1} and  \eqref{finf2} are obtained.
  After integrations by parts and following the same lines of the proof for the above estimates, by using \eqref{eq.freest4}, \eqref{eq.freest5} and \eqref{eq.freest6} we finally get \eqref{finf3} and \eqref{finf4}.
We notice that, the meromorphic extension of the cutoff resolvent $\varphi (P-\mu^2)^{-1} \varphi,$
guarantees that  the estimates for $\Im \mu>0,$ remain valid also in the domain $\Im \mu>-\delta.$
\end{proof}
\begin{remark}\label{eq.reson} Since the operator $-\Delta -W(|x|)$ has no resonances, the operator $R^{+}(\mu)$ has the form,
\begin{equation}\label{eq.projres}
R^{+}(\mu) =\sum_{j>0}\frac{B_j}{(\mu-\mu_j)}+Q(\mu)
\end{equation}
where the $B_j$ are projection operators on the eigenspaces associated to the eigenvalues
$\mu=\mu_j,$ while $Q(\mu)$ is analytic in $\Im \mu>-\delta.$. The operator
$$ P_{ac}R^{+}(\mu) = P_{ac}Q(\mu),$$
 is also analytic in $\Im \mu>-\delta.$
\end{remark}

We now finally give the

\begin{theo}\label{th.localenergy}
Let  $u_0$ and $v_0$ are Schwartz functions and $\varphi(x)=e^{-\delta|x|},$ for some $\delta>0.$ Then there exists $a>0,$ such that the solution to \eqref{NLW}, assuming that the potential $W$ has the properties  \eqref{ass.res2}, satisfies the following,
\begin{equation}\label{eq.localen1}
\norm{\varphi P_{ac}u(t,x)}_{L^2}\lesssim e^{-at}(\norm{u_0}_{L^2}+\norm{v_0}_{H^{-1}}),
\end{equation}
and
\begin{equation}\label{eq.localen2}
\norm{\varphi \partial_t P_{ac}u(t,x)}_{L^2}\lesssim e^{-at}(\norm{u_0}_{H^1}+\norm{v_0}_{L^2}).
\end{equation}
\end{theo}

\begin{proof}
The proof of the main theorem follows the one of Vainberg in \cite{Vai89} and \cite {Sj2001}.
\emph{Proof of \eqref{eq.localen1}.} We have no loss of generality if we assume $u(0,x)=0$ in \eqref{NLW} and $u,v_0$ are real -- valued functions. Let be $\mathcal{H}$ a general Hilbert space and denote by $L_{\nu}(\rone, \ch)=e^{-\nu r}L(\rone, \ch).$ We have, by \eqref{eq.laplfree}, the inversion formula
\begin{equation}\label{eq.invfor}
u(t,x)=\frac{1}{2\pi}\Re \int_{-\infty+i\nu}^{\infty+i\nu} e^{-i\mu t}R^+(\mu)v_0d\mu,
\end{equation}
and the above integral converges in $L_{\nu}(\rone, L^2).$ We indicate by $\overline{R}^+(\mu)$ the meromorphic extension of the cutoff resolvent $\varphi R^+(\mu)\varphi.$
We indicate by
\begin{equation}\label{eq.resnew}
\widetilde{R}^{+}(\mu) =\overline{R}^+(\mu)-\frac{\varphi^2}{\mu-i(\nu-1)}.
\end{equation}
From the fact that
\begin{equation}\label{eq.termzero}
\Re \int_{-\infty+i\nu}^{\infty+i\nu} \frac{e^{-i\mu t}}{\mu-i(\nu-1)}d\mu=0,
\end{equation}
we can rewrite the \eqref{eq.invfor} as
\begin{equation}\label{eq.invfor2}
\varphi u(t,x)=\frac{1}{2\pi} \Re \int_{-\infty+i\nu}^{\infty+i\nu} e^{-i\mu t}\widetilde{R}^{+}(\mu)\varphi ^{-1}v_0d\mu.
\end{equation}

\begin{figure}[htb!]
\centering%
\includegraphics[scale=0.5]{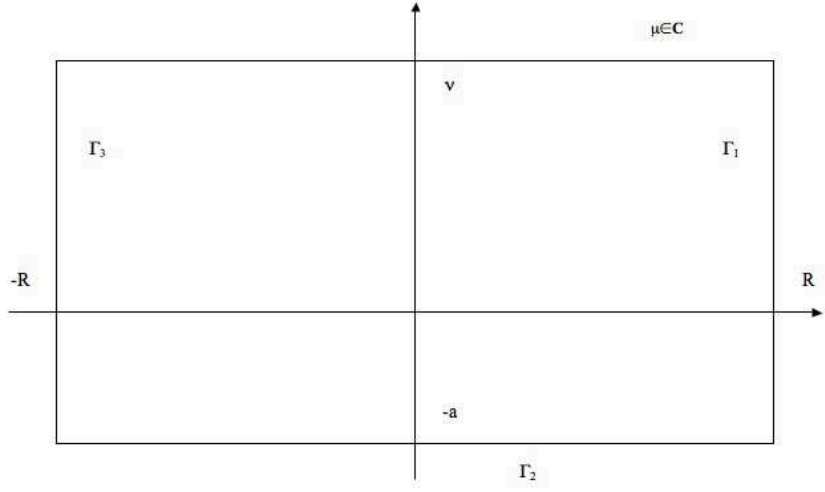}
\caption{The path $\Gamma\cup\Gamma_1\cup\Gamma_2\cup\Gamma_3.$}
\label{fig:FigureExample}
\end{figure}

By Cauchy theorem and integrating along the path showed in the Figure 1 we yield:
\begin{equation}\label{eq.invfor3}
\frac{1}{2\pi}\int_{-\infty+i\nu}^{\infty+i\nu} e^{-i\mu t}\widetilde{R}^{+}(\mu)\varphi ^{-1}v_0d\mu=\sum_{j=1}^{3}\frac{1}{2\pi}\int_{\Gamma_j} e^{-i\mu t}\widetilde{R}^{+}(\mu)\varphi ^{-1}v_0d\mu=\sum_{j=1}^{3}\Lambda_j.
\end{equation}
First we notice that from inequality \eqref{finf2} we achieve
\begin{equation}\label{eq.cont2}
\norm{\Lambda_1}_{L^2}\lesssim\int_{R-ia}^{R+i\nu} \norm{e^{-is t}\widetilde{R}^{+}(s)\varphi ^{-1}v_0}_{L^2}ds\lesssim c(t,\nu)R^{-2}\norm{\varphi ^{-1}v_0}_{H^1}.
\end{equation}
where $c(t,\nu)$ is a positive measurable function depending (exponentially) on $t$ and $\nu.$ A similar estimate is valid also for the term $\Gamma_3.$ Taking $R\rightarrow\infty,$ we easily see that $\norm{\Lambda_l}_{L^2}$ approaches to 0 for $l=1,3.$
As far as concerning the remaining integral $\Lambda_2,$ we obtain
by two integrations by parts,
\begin{equation}\label{eq.invfor4}
\Lambda_2=N(R)\varphi ^{-1}v_0+\frac{1}{2t^2\pi}\int_{-R-ia}^{R-ia} e^{-i\mu t}\frac{d^2}{ds^2}\widetilde{R}^{+}(\mu)\varphi ^{-1}v_0d\mu,
\end{equation}
where we denote by $N(R)$ the expression that includes all the boundary integral terms. If $\overline{\Lambda}_2$ denotes the second term on the right hand side of the previous identity, we may write
\begin{equation}\label{eq.cont3}
\begin{aligned}
\norm{\Lambda_2}_{L^2}\lesssim& \norm{N(R)\varphi ^{-1}v_0}_{L^2}+
\norm{\overline{\Lambda}_2}_{L^2}.
\end{aligned}
\end{equation}
By estimate \eqref{finf1} and \eqref{finf2}, it is easy to se that
the first term on the right hand side of the above identity can be
bounded by $C(R) \norm{\varphi ^{-1}v_0}_{H^{1}},$ where $C(R)$ is
a constant depending on $R$ approaching to 0 as $R\rightarrow0.$
The remaining term can be handled, for $t$ large enough, as
\begin{equation}\label{eq.cont6}
\begin{aligned}
\norm{\overline{\Lambda}_2}_{L^2}\lesssim&
\int_{-R-ia}^{R-ia} e^{-at}\norm{\widetilde{R}^{+}(s)\varphi ^{-1}v_0}_{L^2}ds\lesssim e^{-at}\norm{\varphi ^{-1}v_0}_{H^{-1}}\int_{-R}^{R} \langle{s\rangle}^{-2}ds\lesssim\\
&\lesssim e^{-at}\norm{\varphi ^{-1}v_0}_{H^{-1}},
\end{aligned}
\end{equation}
and this completes the proof of the first part of the theorem.\\
\emph{Proof of \eqref{eq.localen2}.} It is enough to see that the function $\partial_tu(t,x),$ satisfies
the equation \eqref{NLW}, with initial data $v_0$ and $\Delta v_0.$ Now we use the estimate \eqref{eq.localen1}.
\end{proof}

\section{The Nonlinear Schr\"odinger equation}
Consider the   nonlinear Schr\"odinger equation (NLS),
\begin{equation}\label{NLS}
\begin{aligned}
& iu_{t }(t,x)=-\Delta u(t,x) - |u|^{p-1} u(t,x)=0 ,\quad (t,x) \in \mathbf{R}\times\mathbf{R}^3,
\\&
  u(0,x)=u_0(x),
 \end{aligned}
\end{equation}
where $1<p<1+4/3,$ this means in the domain where the problem is globally well-posed. In fact for $p>1+4/3,$ can be exist solutions with $H^1$ norms blowing up in a finite time interval.
Solitary waves associated with the NLS type equation have the form
$$
\psi_s (t,x) = \chi_{\omega}(x) e^{-i \omega t},\quad t\in \mathbf{R}, x\in \mathbf{R}^3,
,$$
with $\omega\subset \mathcal{O},$ for some open interval $\mathcal{O}\in \mathbf{R},$ where $\chi_{\omega} $  satisfies the equation
\begin{eqnarray}\label{eq.M-S1}
 -\Delta \chi_{\omega} -\omega \chi_{\omega} +|\chi_{\omega}|^{p-1} \chi_{\omega}=0     ,\quad x \in \mathbf{R}^3, \\
\int_{\mathbf{R}^3} \chi^2 =1.\label{eq.M-S3}
\end{eqnarray}
We recall some well known facts about the linearization at a ground
state. Let us write the ansatz
\begin{equation} \label{eq:decomp1}
 u(t,x) = e^{i t}
(\chi _{\omega} (x)+
r(t,x)),
\end{equation}
Inserting it into the equation \eqref{eq:decomp1} we get
\begin{equation}\label{eq:linearized}
\begin{aligned} &
  i \partial_{t}r =
 \mathcal{A} r +(\text{nonlinear terms}),
 \end{aligned}
\end{equation}
and
$$
\mathcal{A}(\omega)r=(-\Delta +\omega-\tfrac{p+1}{2}\chi_\omega^p)r-\tfrac{p-1}{2}\chi_\omega^p\overline{r}.
$$

Because of the presence of the variable $\overline{r}$, we write the above as a system. This yields to
\begin{equation}\label{eq:Linearmat}
\begin{aligned} &
  i \partial_{t}R =
 H(\omega) R
 \end{aligned}
\end{equation}
where
\begin{equation} \label{eq:operator}
\begin{aligned}
&H(\omega)=\begin{pmatrix}0 &
L_+(\omega)  \\
-L_-(\omega)& 0,
 \end{pmatrix} \end{aligned}
\end{equation}
and with
\begin{equation} \label{eq:operator2}
\begin{aligned}
L_+(\omega)=-\Delta +\omega-p\chi_\omega^p, \qquad L_-(\omega)=-\Delta +\omega-\chi_\omega^p,
\end{aligned}
\end{equation}
having in mind also that the essential spectrum of $H(1)$ consists of
$(-\infty , -\omega]\cup [ \omega,+\infty ),$ and that
$0$ is its isolated eigenvalue. Furthermore
 it is easy to
see that $L(\omega)_{\pm}$ are self-adjoint operator with continuous spectrum in $[\omega, \infty)$, that
$L(\omega)_{-}$ is nonnegative, while $L(\omega)_{+}$ has exactly one negative eigenvalue (see the paper \cite{CGNT2007} and \cite{We85} for more details).
We get, according also to the results in \cite{DS2006}, the following:

\begin{theo}\label{Th:NLSres}
The operators $L_\pm(\omega)$ have neither an eigenvalue nor a strong resonance at spectral point $\omega.$ Moreover the linearized operator $H(\omega)$ has no strong resonances at the spectral points $\pm\omega.$
\end{theo}


Before to start the proof of the above theorem we need to give some preliminary lemmas.
By a rescaling argument we can pick $\omega=1,$ and focalize our attention on the operator
$$L_-= - \Delta \chi_1 +  \chi_1 - \chi_1^p , \ \chi_1 \in H^2(R^3), \ \chi_1>0 $$ because all results can be proved in the same manner for $L_{+}.$
It is well - known that positive radial solutions exist and they are exponentially decaying (see, for istance, \cite{GNN81}).
Here we briefly sketch the proof for completeness and make a better asymptotic expansion.
First we note that the Sobolev embedding implies
\begin{equation}\label{eq.sob11}
    |\chi_1(x)| \leq C.
\end{equation}

A better decay estimate follows from an argument of Strauss (see page 155, section 2 of
\cite{St77}). The classical Strauss lemma (Radial Lemma 1 in \cite{St77}) gives
$$ r \chi_1(r) \leq \|\chi_1\|_{H^1}= C,$$
so $$ \chi_1(r) - \chi_1^p(r) > (1-\delta) \chi_1(r)$$
for any positive $\delta$ and for $r>0$ large enough.
Setting
$u=\chi_1(r) r,$ we see that $ u^{\prime \prime}(r) \geq (1-\delta) u(r)$ so
$$ \left( \frac{u^2}{2} \right)^{\prime \prime} \geq \left( u^\prime \right)^2 + (1-\delta) u^2 \geq (1-\delta) u^2.$$
This differential inequality shows that the quantity
$$ e^{-\sqrt{2(1-\delta)}r} (w^\prime + \sqrt{2(1-\delta)} w), \ \ w(r) = u^2(r) $$
in a non-decreasing  and has to be non -- negative, since $w$ and $w^\prime$ are integrable on $(0,\infty).$
This implies
 the decay estimate
$$ u^2(r) \leq C e^{-\sqrt{2(1-\delta)}r} , \ \ \ r \rightarrow \infty.$$
Using the argument of Remark \ref{r.lio}, we arrive at the following.
\begin{lem} \label{l.AS1}
For any $\omega >0 $  there exists a
unique positive solution $\chi(x) = \chi_\omega (x)  \in H^1$ of
the equation \eqref{eq.M-S1} and a positive $\delta_0=\delta_0(\omega)$,  such that
$$ |\chi_\omega(r)| \leq C e^{-\delta_0 r}.$$
\end{lem}

To obtain more precise estimate we rewrite
$$ - \Delta \chi_1 +  \chi_1 = F(r) \equiv \chi_1^p , \ \chi_1 \in H^1(R^3), \ \chi_1>0 $$
as follows
$$ \chi_1(|x|) = c \int_{\mathbb{R}^3} \frac{e^{-|x-y|}}{|x-y|} F(|y|) dy. $$
Introducing polar coordinates, we find
$$ \chi_1(|x|) = c \int_0^\infty \int_{\mathbb{S}^2} \frac{e^{-r}}{r} F(|x+r\omega|) d\omega r^2 dr. $$
Now we can use the following identity
$$ \int_{\mathbb{S}^2} F(|x+r\omega|) d\omega = \frac{c}{|x|r} \int_{||x|-r|}^{|x|+r} F(\lambda) \lambda d\lambda$$
so
$$ \chi_1(|x|) = \frac{c}{|x|} \int_0^\infty  e^{-r} \int_{||x|-r|}^{|x|+r} F(\lambda) \lambda d\lambda dr. $$
One can see that
\begin{equation}\label{eq.k1andk2}
   \chi_1(|x|) = \frac{c}{|x|} \left(K_1(F)(|x|) +  K_2(F)(|x|) \right)
\end{equation}
with
$$ K_1(F)(|x|) = \int_0^{|x|}  e^{-|x|} \sinh (\lambda)  F(\lambda) \lambda d\lambda.$$
$$ K_2(F)(|x|) = \int_{|x|}^\infty \sinh (|x|) e^{-\lambda}   F(\lambda) \lambda d\lambda.$$
If one substitutes $F(\lambda)$ with $\chi^p_1(\lambda)$
and note that $F(y) = \chi_1^p(y) \leq C $ for $|y|$ bounded due to \eqref{eq.sob11} and moreover the
estimate of Lemma \ref{l.AS1} implies $F(y) \leq C e^{-\delta_0 p |y|}$ for $|y| \geq 1$ so we deduce
$$ |K_1(F)(|x|)| \leq C e^{-B|x|}, \ \ B = \min (1, \delta_0 p)$$
$$ |K_2(F)(|x|)| \leq C e^{-B|x|}, \ \ B = \min (1, \delta_0 p)$$
so we find
$$|x| |\chi_1(|x|)| \leq C e^{-B|x|}, \ \ B = \min (1, \delta_0 p).$$
If $\delta_0p < 1$ we derive $F(|x|) \leq C e^{-pB|x|}$  so making further iterations we get:

\begin{lem} \label{l.AS2}
For any $\omega >0 $  there exists a
unique positive solution $\chi(x) = \chi_\omega (x)  \in H^1$ of
the equation \eqref{eq.M-S1} and a positive $C=C(\omega)$,  such that
$$ r|\chi_\omega(r)| \leq C e^{-\sqrt{\omega} r}$$
for $r>0.$
\end{lem}

To get asymptotic expansion, we use \eqref{eq.k1andk2}
If one substitutes $F(\lambda)$ with $\chi^p_1(\lambda)$ and apply the estimate of Lemma \ref{l.AS2} (with $\omega=1$)
one can obtain the asymptotic expansions
$$ K_1(\chi_1^p)(|x|) = C_0 e^{-|x|} + O \left(e^{-(p-\delta)|x|} \right), \ \ \ |x| >1,$$
$$ K_1^\prime(\chi_1^p)(|x|) = -C_0 e^{-|x|} + O \left(e^{-(p-\delta)|x|} \right),\ \ \ |x| >1,$$
$$ K_2(\chi_1^p)(|x|) =  O \left(e^{-(p-\delta)|x|} \right),\ \ \ |x| >1,$$
$$ K_2^\prime(\chi_1^p)(|x|) =  O \left(e^{-(p-\delta)|x|} \right),\ \ \ |x| >1,$$
where $\delta$ is any positive number. After rescaling argument we get.

\begin{lem} \label{l.AS3}
For any $\omega >0 $  there exists a
unique positive solution $\chi(|x|) = \chi_\omega (x)  \in H^2$ of
the equation \eqref{eq.M-S1} and a positive $C_0=C(\omega)$,  such that
$$ r\chi_\omega(r) = C_0 e^{-\sqrt{\omega}r} + O \left(e^{-(p-\delta)\sqrt{\omega}r} \right), \ \ \ r>1,$$
$$ \left( r\chi_\omega(r) \right)^\prime = - \sqrt{\omega} C_0 e^{-\sqrt{\omega}r} + O \left(e^{-(p-\delta)\sqrt{\omega}r} \right),  \ \ \ r>1.$$
\end{lem}

\begin{theo} \label{th.AS3}
For any $\omega >0 $  there exists a
unique positive solution $\chi(|x|) = \chi_\omega (x)  \in H^1$ of
the equation \eqref{eq.M-S1}
so that
$$ 0 > \frac{\partial_{|x|}\chi_\omega(x)}{\chi_\omega(x)} > -C_1$$
for some positive constant $C_1.$ Moreover there exists
a positive $C_0=C(\omega)$ constant,  such that
$$ r\chi_\omega(r) = C_0 e^{-\sqrt{\omega}r} + O \left(e^{-(p-\delta)\sqrt{\omega}r} \right), \ \ \ r>1,$$
$$ \left( r\chi_\omega(r) \right)^\prime = - \sqrt{\omega} C_0 e^{-\sqrt{\omega}r} + O \left(e^{-(p-\delta)\sqrt{\omega}r} \right),  \ \ \ r>1.$$
\end{theo}

\begin{proof}[Proof of Theorem \ref{Th:NLSres}]
The proof is easy, so we reduce it in few lines and it is a consequence of the results stated in the previous sections. By Theorem \ref{Th:NLSres}, we obtain that the operator $L_{-}-1$ has the form
of $-\Delta-W,$ where $W$ satisfies the assumption of Theorem \ref{th:noresR3} and this assure that we have no resonance (or eigenvalue) at zero energy, the same is valid for $L_{+}-1$. Finally, using Lemma 16 in \cite{Sch09} we achieve the proof for $H(1).$
\end{proof}

\section{Appendix 1: Asymptotics for solutions to some ODE}

\begin{lem} \label{lem:2}  Assume $ W(s)>0 , \forall s >0$ and
$$(1+s)^2 W(s) \in L^1(0,\infty) .$$
If   $u\in C^2(\mathbb R)$ is a  solution of $u''+W u=0$ and there exist $s_0 \geq 0$ so that
  $u(s) > 0$ for $s > s_0,$
then there exits a non - negative number $C_1$ so that we have the relations
\begin{equation}\label{eq.idin1}
    u^\prime(s) = C_1+ \int_s^\infty W(\tau) u(\tau) d\tau,
\end{equation}
\begin{equation}\label{eq.idin2}
    u(s) = u(s_0) + \int_{s_0}^\infty(\tau-s_0) W(\tau) u(\tau) d\tau + C_1 (s-s_0) - \int_s^\infty (\tau - s) W(\tau) u(\tau) d\tau,
\end{equation}
as well as the asymptotic expansions (valid for $s \rightarrow \infty$)
$$ u(s) = C_0 + C_1 s + O( g(s)), \ \ u^\prime(s) = C_1+  O( h(s)),$$
where
$$ h(s) = \int_s^\infty (1+s)W(s) ds, \ \ g(s) = \int_s^\infty (1+s)^2 W(s) ds. $$
\end{lem}
\begin{proof}We have $$ u^{\prime \prime}(s) = - W(s)u(s) < 0, \ \ s \geq s_0.$$ The Taylor expansion gives
$$0<u(s) = D_0 + D_1 s + \frac{u^{\prime \prime}(\xi)(s-s_0)^2}{2} < D_0 + D_1 s,$$
where $$s> s_0, \xi \in (s_0, s), \ \ D_1= u^\prime(s_0),\  D_0=u(s_0) -  s_0 u^\prime(s_0).$$
The inequality
$$0 < u(s) < C(1+s)$$ and the assumption
$$(1+s)^2 W(s) \in L^1(0,\infty) $$ show that
\begin{equation}\label{eq.basinn}
   u^\prime(s)-u^\prime(t) = -\int_t^s W(\tau) u(\tau) d \tau
\end{equation}
is small when $ s_0 < t < s$ and $s,t$ are large enough. This argument shows the existence of the limit
$$ \lim_{s \rightarrow \infty} u^\prime(s) = C_1$$
as well as the asymptotic expansion
$$ u^\prime(s) = C_1+  O( h(s)).$$
Integrating this relation and using the fact
$$ \int_s^\infty h(\tau) d\tau = O( g(s)),$$
we obtain the desired expansion
$$ u(s) = C_0 + C_1 s + O( g(s)).$$
The fact that $C_1 \geq 0$ follows from the positivity of $u(s)$ for $s > s_0.$
Finally, to prove  (\ref{eq.idin2}) we use (\ref{eq.basinn}) and integrating (\ref{eq.idin2}) we find
(\ref{eq.idin1}).

This completes the proof.
\end{proof}

A slight modification is the following.

\begin{lem} Assume $ W(s)>0 , \forall s >0$ and
$$(1+s) W(s) \in L^1(0,\infty) .$$
If   $u\in C^2(\mathbb R)$ is a  solution of $u''+W u=0$ and there exist $s_0 > 0$ so that
  $u(s) > 0$ for $s > s_0,$
then
$$ u(s) \leq C(1+s)$$ and the limit
$$ \lim_{s \rightarrow \infty} u^\prime(s)$$
 exists.
\end{lem}
\begin{proof}As in the proof of the previous Lemma we have $$ u^{\prime \prime}(s) = - W(s)u(s) < 0.$$ The Taylor expansion gives
$$0<u(s) = D_0 + D_1 s + \frac{u^{\prime \prime}(\xi)(s-s_0)^2}{2} < D_0 + D_1 s,$$
where $$s> s_0, \xi \in (s_0, s), \ \ D_1= u^\prime(s_0),\  D_0=u(s_0) -  s_0 u^\prime(s_0).$$
The inequality
$$0 < u(s) < C(1+s)$$ and the assumption
$$(1+s) W(s) \in L^1(0,\infty) $$ show that
$$ u^\prime(s)-u^\prime(t) = -\int_t^s W(\tau) u(\tau) d \tau $$
is small when $ s_0 < t < s$ and $s,t$ are large enough. This argument shows the existence of the limit
$$ \lim_{s \rightarrow \infty} u^\prime(s) .$$
This completes the proof.
\end{proof}

Our next step is to consider the equation \eqref{eq.Pu} with potential
  $W(r) = W_1(r)+W_2(r),$ where
\begin{equation}\label{eq.w1a}
    W_1(r) = -\frac{\alpha (\alpha + 1)}{r^2},  \quad \alpha \geq 0,
\end{equation}
while $W_2(r) $ is a $C^1(0, \infty)$ positive strictly decreasing function such that for some positive constants $C, \varepsilon_0$ satisfies the estimate \eqref{eq.h2b}.

\noindent The first step is to obtain asymptotic expansions of the solution and for this aim, by the Definitions \ref{eq.h2sob} and \ref{defreslistrong}, we give the following lemma.

\begin{lem} \label{l.asex} Suppose  \eqref{eq.h2b} is true.

\noindent a) If $0$ is  an eigenvalue  of $P$ and $Pu=0,$ in sense of Definition \ref{eq.h2sob}, then one can find a real number $C$ so that
\begin{equation}\label{eq.O7las}
    u(r) =  \frac{C}{r^\alpha} + O \left(e^{-\varepsilon_0 r/2} \right)
\end{equation}
and
\begin{equation}\label{eq.O8las}
    u^\prime (r) =  - \frac{C \alpha}{r^{\alpha+1}} + O \left(e^{-\varepsilon_0 r/2} \right)
\end{equation}
as $r \rightarrow \infty$.

\noindent b) If $0$ is  a strong  resonance of $P$ and $Pu=0$ in sense of Definition \ref{defreslistrong}, then there exists a real number $C>0$ so that
\begin{equation}\label{eq.O7l}
    u(r) =  \frac{C}{r^\alpha} + O \left( e^{-\varepsilon_0 r/2} \right)
\end{equation}
and
\begin{equation}\label{eq.O8l}
    u^\prime (r) =  - \frac{C \alpha}{r^{\alpha+1}} + O \left(e^{-\varepsilon_0 r/2} \right)
\end{equation}
as $r \rightarrow \infty$.
\end{lem}

\begin{proof} First we prove a). One can rewrite the equation $Pu=0$ as
\begin{equation}\label{eq.O1as}
    \left[r^{-\alpha} \left( u'(r) + \frac{\alpha}r  u(r)\right)   \right]' + r^{-\alpha} W_2(r)u(r)=0\,
\end{equation}
or as
\begin{equation}\label{eq.O1as007}
    \left[r^{-2\alpha} \left( r^\alpha u(r) \right)^\prime   \right]' + r^{-\alpha} W_2(r)u(r)=0\, .
\end{equation}
Note that the assumption $W_2(r) \in C^1(0,\infty)$
combined with the equation $Pu=0$ imply that $u \in C^2(R_1, R_2)$ for any $0 < R_1 < R_2.$ Integrating \eqref{eq.O1as007} in the interval $(R_1, R_2),$ we find
\begin{equation}\label{eq.O2as} \begin{split}
    \left| R_2^{-2\alpha} ( R_2^{\alpha} u(R_2) )' - R_1^{-2\alpha} ( R_1^{\alpha} u(R_1) )' \right| = \\ = \left| \int_{R_1}^{R_2} \tau^{-\alpha}
    W_2(\tau) u(\tau) \dif\tau \right|,
\end{split}\end{equation}
so using the assumption \eqref{eq.h2b} together with the fact that $u$ is bounded (since it belongs to $H^1(1,+\infty)\cap C^1 (1,+\infty)$), and taking $R_1>1,$ we find
\begin{equation}\label{eq.int1}
\left| \int_{R_1}^{R_2} \tau^{-\alpha}
    W_2(\tau) u(\tau) \dif\tau \right| \leq C e^{-\varepsilon_0 R_1/2}.
    \end{equation}
   In this way we conclude that the limit
    \begin{equation}\label{eq.O4}
       \lim_{r  \rightarrow \infty}  r^{-2\alpha} \left( r^\alpha u(r) \right)^\prime,
    \end{equation}
exists and it is equal to a real constant $C_0.$
By this, we achieve the expansion
\begin{equation}\label{eq.O5as}
    \left( r^\alpha u(r) \right)' = C_0 r^{2\alpha} +  O \left( e^{-\varepsilon_0 r/2} \right).
\end{equation}
Consider now the function
\[ g(r) = r^\alpha u(r) - C_0 \frac{r^{2\alpha + 1}}{2\alpha+1},\]
then \eqref{eq.O5as} implies that $g^\prime(r)  \in L^1(1,\infty).$ Moreover we can see that $g(r)$ has a limit (say $C$) as $r$ goes to $\infty$ and
\[ g(r) = C - \int_r^\infty g'(\tau) \dif\tau = C + O( e^{-\varepsilon_0 r/2} ).\]
Thus we obtain
\begin{equation}\label{eq.O7as}
    u(r) = C_0 \, \frac{r^{\alpha+1}}{2\alpha + 1} + \frac{C}{r^\alpha} + O \left( e^{-\varepsilon_0 r/2}  \right)
\end{equation}
and
\begin{equation}\label{eq.O8as}
    u' (r) = C_0 \, \frac{(\alpha+1) r^{\alpha}}{2\alpha + 1} - \frac{C \alpha}{r^{\alpha+1}} + O \left(e^{-\varepsilon_0 r/2} \right)
\end{equation}
as $r$ move to infinity. Comparing these asymptotic expansions with the fact that $u$ is bounded, we see that $C_0=0$ and this completes the first part of the lemma.

The proof of b) can be obtained similarly to the above  using the assumption \eqref{eq.1asust}, so we skip it.
\end{proof}
The above arguments suffices to get
\begin{lem} \label{l.oneresweak}
If $0$ is  a weak resonance of $P$ and $Pu=0$, then one can find  real numbers $C_0, C_1$ so that
\begin{equation}\label{eq.O7leakw}
    u(r) =  C_0 \, \frac{r^{\alpha+1}}{2\alpha + 1} + \frac{C_1}{r^\alpha} + O \left(e^{-\varepsilon_0 r/2} \right)
\end{equation}
and
\begin{equation}\label{eq.O8lweak}
    u^\prime (r) = C_0 \, \frac{(\alpha+1) r^{\alpha}}{2\alpha + 1} - \frac{C_1 \alpha}{r^{\alpha+1}} + O \left(e^{-\varepsilon_0 r/2}\right)
\end{equation}
as $r >1$ tends to infinity.
\end{lem}

\section{Appendix 2}

In this Section we complete the discussion concerning the weak resonances and its connection with
different type of potentials. It seems that the weak resonances cannot be never avoid, more precisely
they are a intrinsic character of the structure of the differential equation involved in the description of such phenomena. In order to do that we look at large potentials and small potentials.

\subsection{Large potentials do not generate weak resonance at the spectral origin.}

As in the previous section we shall assume $ W(s)>0 , \forall s >0$ and
$$(1+s)^2 W(s) \in L^1(0,\infty) .$$

To show that all solutions    $u\in C^2(\rone)$ to $u''+W u=0$ having linear growth at infinity
are identically zero, we can apply Lemma \ref{lem:2} so without loss of generality one can assume
$$ u(0)=0, u(s)>0 ,\ \forall s >0.$$
The key assumption that will guarantee that such solutions do not exist is the following one
\begin{equation}\label{eq.LPNR}
   \frac{1}{M} \int_0^M s^2 W(s) ds \geq 1
\end{equation}
for some real $M>0.$

Turning back to the integral equation of Lemma \ref{lem:2} we have the following relations
\begin{equation}\label{eq.idin1asy}
    u^\prime(s) = C_1+ \int_s^\infty W(\tau) u(\tau) d\tau,
\end{equation}
\begin{equation}\label{eq.idinsy2a}
    u(s) =  \int_{0}^\infty \tau W(\tau) u(\tau) d\tau + C_1 s - \int_s^\infty (\tau - s) W(\tau) u(\tau) d\tau,
\end{equation}
so we can introduce the operator
$$ K(u)(s) = \int_{0}^\infty \tau W(\tau) u(\tau) d\tau - \int_s^\infty (\tau - s) W(\tau) u(\tau) d\tau =$$
$$ \int_{0}^s \tau W(\tau) u(\tau) d\tau + s \int_s^\infty  W(\tau) u(\tau) d\tau.$$
It is clear that $u \geq 0, \ \forall s \geq 0$ implies $K(u)(s) \geq 0 \ \forall s \geq 0$.
The relation (\ref{eq.idin1asy}) show that $u^\prime(s)>0$ so for the interval
$[0,M]$ so that $u^\prime(s) \geq C^* = u^\prime(M)  > 0 $ for $ 0 \leq s \leq M.$ Then we have the estimate
$$ u(s) \geq C^* s , \ \ \forall s \in [0,M].$$
For any $s \in [0,M]$ we have
$$K(u)(s) > K_M(u)(s) \equiv \int_{0}^s \tau W(\tau) u(\tau) d\tau + s \int_s^M  W(\tau) u(\tau) d\tau $$
so
$$ \frac{K(u)(s)}{s} \geq C_*  \left(\frac{1}{s}\int_{0}^s \tau^2 W(\tau) d\tau + \int_s^\delta \tau W(\tau)  d\tau \right).$$
The function
$$ \frac{1}{s}\int_{0}^s \tau^2 W(\tau) d\tau + \int_s^M \tau W(\tau)  d\tau $$ is decreasing so
$$ \frac{K(u)(s)}{s} \geq C_* \left( \frac{1}{M}\int_{0}^M \tau^2 W(\tau) d\tau \right) .$$
The assumption (\ref{eq.LPNR}) implies
$$ D = \frac{1}{M}\int_{0}^M \tau^2 W(\tau) d\tau \geq 1.$$
so if $u$ solves the equation $u(s) = C_1s + K(u)(s)$ we can interating the estimate
$$ \frac{u(s)}{s} \geq C^*  \ \forall s \in [0,M] \Longrightarrow  \frac{u(s)}{s} = C_1 + C^* D > C^* \ \forall s \in [0,M],$$
we find
$$\lim_{s \rightarrow 0_+}\frac{u(s)}{s} = u^\prime(0) \geq C^* + N C_1 ,$$ but this is a contradiction since  $N$ is arbitrary.

\subsection{Small  potentials  generate weak resonance at the spectral origin.}

To construct nontrivial solutions    $u\in C^2(\mathbf R_+)$ to $u''+W u=0$ having linear growth at infinity, we assume
\begin{equation}\label{eq.LPER}
 \sum_{M>0} \frac{1}{M}\int_0^M s^2 W(s) ds = D  < 1.
\end{equation}

\begin{remark}
One sufficient condition for (\ref{eq.LPER}) is
$$\int_0^\infty s W(s) ds < 1, $$
since we have the estimate
$$ \frac{1}{M}\int_0^M s^2 W(s) ds \leq \int_0^M s W(s) ds.$$
\end{remark}

Consider the integral equation
$$ u(s) = C_1 s + K(u)(s), $$
where
$$ K(u)(s) = \int_{0}^\infty \tau W(\tau) u(\tau) d\tau - \int_s^\infty (\tau - s) W(\tau) u(\tau) d\tau =$$
$$ \int_{0}^s \tau W(\tau) u(\tau) d\tau + s \int_s^\infty  W(\tau) u(\tau) d\tau.$$

We plan to show that this equation has a solution in the Banach space obtained as a closure of the linear space $L$ formed by the functions $u(s) \in C([0, \infty)),$ such that $u(0)=0$ and
$$ \|| u \|| = \sup_{s > 0 } \frac{|u(s)|}{s} < \infty.$$
To be more precise, $B$ is the closure of $L$ with respect to the norm $ \|| u \||.$

To show this fact it is sufficient to notice that
$$ \|| K(u) \|| \leq D \|| u \||$$
so the assumption $D<1$ enables one to apply a contraction argument for the equation
$$ u(s) = C_1 s + K(u)(s). $$

\section*{Acknowledgements}
The first  author was supported by the Italian National Council of Scientific Research (project PRIN No.
2008BLM8BB)
entitled: "Analisi nello spazio delle fasi per E.D.P."\\
The second author is supported by an INdAM grant. Currently he is a Academic Visitor at Department of Mathematics of the Imperial College London.

\bibliographystyle{amsalpha} 


\end{document}